\documentclass[10pt,twoside]{article}
\usepackage{mathrsfs}
\usepackage{amssymb}
\usepackage{amsmath,color}
\usepackage[all]{xy}
\usepackage{amsthm}
\usepackage{url}
\numberwithin{equation}{section}

\newtheorem{theorem}{Theorem}[section]
\newtheorem{proposition}[theorem]{Proposition}
\newtheorem{definition}[theorem]{Definition}
\newtheorem{remark}[theorem]{Remark}
\newtheorem{lemma}[theorem]{Lemma}
\newtheorem{example}[theorem]{Example}

\newtheorem{theorem*}{Theorem}

\newcommand{\modu}{\operatorname{mod}}

\newcommand{\Gen}{\operatorname{Gen}}
\newcommand{\Cogen}{\operatorname{Cogen}}
\newcommand{\Hom}{\operatorname{Hom}}

\newcommand{\Ext}{\operatorname{Ext}}
\newcommand{\inj}{\operatorname{inj}}

\newcommand{\add}{\operatorname{add}}

\newcommand{\pd}{\operatorname{pd}}

\newcommand{\proj}{\operatorname{proj}}
\newcommand{\Ker}{\operatorname{Ker}}
\newcommand{\Coker}{\operatorname{Coker}}
\newcommand{\Ima}{\operatorname{Im}}
\newcommand{\RNum}[1]{\uppercase\expandafter{\romannumeral #1\relax}}

\title{ \bf Recollements and tilting modules 
}
\vspace{0.2cm}

\author{Xin Ma$^1$, Tiwei Zhao$^2$\thanks{Corresponding author. E-mail address: tiweizhao@qfnu.edu.cn}
\\
{\it \footnotesize 1. College of Science, Henan University of Engineering, Zhengzhou 451191, P.R. China}
\\
{\it \footnotesize 2. School of Mathematical Sciences, Qufu Normal University, Qufu 273165, P.R. China}
\\
}
\date{ }
\begin{document}

\baselineskip=16pt
\maketitle

\begin{abstract}
Let $(\modu \Lambda',\modu \Lambda,\modu \Lambda'')$ be a recollement of module  categories for artin algebras $\Lambda'$, $\Lambda$ and $\Lambda''$.
 We provide a sufficient condition such that a glued torsion pair in $\modu \Lambda$ is tilting when the given two torsion pairs are tilting  in $\modu \Lambda'$ and $\modu \Lambda''$ respectively.
Using this result, we give a construction of gluing of tilting modules
in $\modu \Lambda$ with respect to tilting modules in $\modu \Lambda'$ and $\modu \Lambda''$ respectively.
\end{abstract}

\pagestyle{myheadings}
\markboth{\rightline {\scriptsize   X. Ma and T. Zhao}}
         {\leftline{\scriptsize Recollements and tilting modules}}

\section{Introduction} 

Recollements of abelian and triangulated categories were introduced by Be{\u\i}linson, Bernstein and Deligne \cite{BBD81F} in
connection with derived categories of sheaves on topological spaces with the idea that one triangulated category may be ``glued together" from two others.

Recently, gluing techniques with respect to a recollement of triangulated or abelian categories have been investigated for cotorsion pairs \cite{CJM13C}, torsion pairs \cite{MXHZY}, and so on (e.g. \cite{MZH,Qin,Zhang,Zheng}).
In particular, for a recollement of triangulated categories, Liu, Vit\'oria and Yang presented explicit constructions of gluing of silting objects \cite{LQHVJ14G}. Moreover, Zhang considered the recollement of wide subcategories of abelian categories using gluing techniques \cite{Zhang}.

The classical tilting modules were introduced by Brenner and Butler \cite{BB79G}, and Happel and Ringel \cite{HDRCM82T}.
 It is closely related with torsion pairs; for instance, a tilting module can induce a torsion pair, and conversely, a torsion pair satisfying certain conditions can induce a tilting module (see \cite{AISD}).

 It is natural to ask that whether a tilting module can be ¡°glued together¡± from the other two tilting modules in a recollement of module categories.
For an artin algebra $\Lambda$, we use $\modu \Lambda$ to denote the category of finitely generated left $\Lambda$-modules. Let $\Lambda'$, $\Lambda$ and $\Lambda''$ be artin algebras such that there is a recollement of module categories:
$$\xymatrix{\modu \Lambda'\ar[rr]!R|-{i_{*}}&&\ar@<-2ex>[ll]!R|-{i^{*}}
\ar@<2ex>[ll]!R|-{i^{!}}\modu \Lambda
\ar[rr]!L|-{j^{*}}&&\ar@<-2ex>[ll]!L|-{j_{!}}\ar@<2ex>[ll]!L|-{j_{*}}
\modu \Lambda''}.$$

The paper is organized as follows.

In Section 2, we give some terminologies and some preliminary results.

The first author and Huang have given a construction of tilting modules in $\modu \Lambda$ from the tilting modules in $\modu \Lambda'$ and $\modu \Lambda''$ in \cite[Theorem 3.3]{MXHZYT} under some conditions.
For a module $T\in\modu\Lambda$, we write $T^{\bot_0}:=\{M\in\modu\Lambda\mid \Hom_{\Lambda}(T,M)=0\}$.
In Section 3, we first give the constructions of tilting torsion pairs in the recollement of abelian categories. Using them, we weaken the conditions in \cite[Theorem 3.3]{MXHZYT} and obtain the main result as follows.

\begin{theorem}\label{THM1}
Let $T'$ and $T''$ be tilting modules in $\modu \Lambda'$ and $\modu \Lambda''$ respectively. If $i^{!}$ and $j_{!}$ are exact, then there exists a tilting $\Lambda$-module $T$ forming as $T:=j_{!}(T'')\oplus M$ for some $\Lambda$-module $M$, such that $(\mathcal{T},\mathcal{F})=(\Gen T,T^{\perp_{0}})$, where  $(\mathcal{T},\mathcal{F})$ is a glued torsion pair in $\modu \Lambda$ with respect to $(\Gen T', T'^{\perp_{0}})$ and $(\Gen T'', T''^{\perp_{0}})$.
\end{theorem}

We give some examples to illustrate the obtained result in Section 4.

\section{Preliminaries}\label{pre}
Throughout this paper, all subcategories are full, additive and closed under isomorphisms.

First of all, we recall the notion of recollements of abelian categories.
\begin{definition}{\rm (\cite{VFTP04C})\label{def-rec}
A recollement, denoted by ($\mathcal{A},\mathcal{B},\mathcal{C}$), of abelian categories is a diagram
\begin{equation}\label{re}
  \xymatrix{\mathcal{A}\ar[rr]!R|{i_{*}}&&\ar@<-2ex>[ll]!R|{i^{*}}\ar@<2ex>[ll]!R|{i^{!}}\mathcal{B}
\ar[rr]!L|{j^{*}}&&\ar@<-2ex>[ll]!L|{j_{!}}\ar@<2ex>[ll]!L|{j_{*}}\mathcal{C}}
\end{equation}
of abelian categories and additive functors such that
\begin{enumerate}
\item[(1)] ($i^{*},i_{*}$), ($i_{*},i^{!}$), ($j_{!},j^{*}$) and ($j^{*},j_{*}$) are adjoint pairs.
\item[(2)] $i_{*}$, $j_{!}$ and $j_{*}$ are fully faithful.
\item[(3)] $\Ima i_{*}=\Ker j^{*}$.
\end{enumerate}}
\end{definition}

Let $\mathcal{A}$ be an abelian category and $\mathcal{D}$ a subcategory of $\mathcal{A}$. We denote by $\proj \mathcal{A}$ (resp. $\inj \mathcal{A}$) the subcategory of $\mathcal{A}$ consisting of all projective (resp. injective) objects in $\mathcal{A}$.
We use $\add \mathcal{D}$ to denote the subcategory of $\mathcal{A}$ consisting of direct summands of finite direct sums of objects in $\mathcal{D}$.
We use $P(\mathcal{D})$ to denote the direct sum of one copy of each indecomposable $\Ext$-projective object in $\mathcal{D}$, that is
$$P(\mathcal{D})=\{M\in \mathcal{D}\mid \Ext_{\mathcal{A}}^{1}(M,\mathcal{D})=0\}.$$
We write
$${^{\perp_{0}}\mathcal{D}}:=\{M\in\mathcal{A}\mid\Hom_{\mathcal{A}}(M,\mathcal{D})=0\},$$
$${\mathcal{D}^{\perp_{0}}}:=\{M\in\mathcal{A}\mid \Hom_{\mathcal{A}}(\mathcal{D},M)=0\}.$$

Now, we list some properties of recollements of abelian categories (see \cite{VFTP04C,MXHZY,PC14H,PCSO14G,PCVJ14R}), which will be used in the sequel.
\begin{lemma}\label{lem-rec}
Let ($\mathcal{A},\mathcal{B},\mathcal{C}$) be a recollement of abelian categories as (\ref{re}).
\begin{enumerate}
\item[(1)] $i^{*}j_{!}=0= i^{!} j_{*}$.
\item[(2)] The functors $i_{*}$, $j^{*}$ are exact,  the functors $i^{*}$, $j_{!}$ are right exact, and the functor $i^{!}$, $j_{*}$ are left exact.
\item[(3)] All  natural transformations $\xymatrix@C=15pt{i^{*}i_{*}\ar[r]&1_{\mathcal{A}}}$, $\xymatrix@C=15pt{1_{\mathcal{A}}\ar[r]&i^{!}i_{*}}$, $\xymatrix@C=15pt{1_{\mathcal{C}}\ar[r]&j^{*}j_{!}}$, and $\xymatrix@C=15pt{j^{*}j_{*}\ar[r]&1_{\mathcal{C}}}$ are natural isomorphisms. Moreover, all functors $i^{*}$, $i^{!}$ and $j^{*}$ are dense.
\item[(4)] If $i^{*}$ is exact, then $i^{!}j_{!}=0$; and if $i^{!}$ is exact, then $i^{*}j_{*}=0$.
\end{enumerate}
\end{lemma}

Let $\mathcal{A}$ be an abelian category. We recall the notion of torsion pairs as follows.
\begin{definition}
{\rm (\cite{DSE66A}) A pair of subcategories $(\mathcal{T},\mathcal{F})$ of an abelian category $\mathcal{A}$ is called a \emph{torsion  pair}
if the following conditions are satisfied.
\begin{itemize}
\item[(1)] $\Hom_{\mathcal{A}}(\mathcal{T},\mathcal{F})=0$; that is, $\Hom_{\mathcal{A}}(X,Y)=0$
for any $X\in\mathcal{T}$ and $Y\in\mathcal{F}$.
\item[(2)] For any object $M\in \mathcal{A}$, there exists an exact sequence
\begin{align*}
\xymatrix@C=15pt{0\ar[r]&X\ar[r]&M\ar[r]&Y\ar[r]&0}
\end{align*}
in $\mathcal{A}$ with $X\in \mathcal{T}$ and $Y\in \mathcal{F}$.
\end{itemize}}
\end{definition}

Let $(\mathcal{T},\mathcal{F})$ be a torsion pair in an abelian category $\mathcal{A}$. Then we have
\begin{itemize}
\item[(1)] $\mathcal{T}$ is closed under extensions and quotient objects.
\item[(2)] $\mathcal{F}$ is closed under extensions and subobjects.
\end{itemize}
Moreover, we have $$\mathcal{T}={^{\perp_{0}}\mathcal{F}}\ \ \ \ \text{and}\ \ \ \    \mathcal{F}={\mathcal{T}^{\perp_{0}}}.$$

We need the following easy and useful observations.
\begin{lemma}\label{lem-2.6}{\rm (\cite[Lemma 2.4]{MXZTW} and \cite[Lemma 5.3]{ZJL})}
Let $(\mathcal{A},\mathcal{B},\mathcal{C})$ be a recollement of abelian categories as (\ref{re}). Then we have
\begin{itemize}
\item[(1)] If $i^{*}$ is exact, then $j_{!}$ is exact.
\item[(2)] If $i^{!}$ is exact, then $j_{*}$ is exact.
\end{itemize}
\end{lemma}

\section{Tilting modules in a recollement of abelian categories}\label{tor}
Let $(\mathcal{A},\mathcal{B},\mathcal{C})$ be a recollement of abelian categories as (\ref{re}), and let ($\mathcal{T}',\mathcal{F}'$) and ($\mathcal{T}'',\mathcal{F}''$) be torsion pairs in $\mathcal{A}$ and $\mathcal{C}$
respectively. Following \cite{MXHZY}, there is a torsion pair $(\mathcal{T},\mathcal{F})$ in $\mathcal{B}$ defined by
\begin{align*}
\mathcal{T}&:=\{B\in\mathcal{B}\mid i^{*}(B)\in\mathcal{T}'\ \text{and}\ j^{*}(B)\in\mathcal{T}''\},\\
\mathcal{F}&:=\{B\in\mathcal{B}\mid i^{!}(B)\in\mathcal{F}'\ \text{and}\ j^{*}(B)\in\mathcal{F}''\}.
\end{align*}
In this case, we call $(\mathcal{T},\mathcal{F})$ a \emph{glued torsion pair} with respect to $(\mathcal{T'},\mathcal{F'})$ and $(\mathcal{T''},\mathcal{F''})$.

Now we recall the notion of classical tilting modules as follows.

\begin{definition}{\rm (\cite{BB79G,HDRCM82T})
Let $\Lambda$ be an artin algebra. A module $T\in\modu\Lambda$ is called a \emph{partical tilting} module if the following conditions hold.
\begin{itemize}
\item[(T1)] $\pd _{\Lambda}T\leq 1$.
\item[(T2)] $\Ext_{\Lambda}^{1}(T,T)=0$.
\end{itemize}

A partical tilting module $T$ is called a \emph{tilting} module if it also satisfies the following condition.
\begin{itemize}
\item[(T3)] There exists an exact sequence $\xymatrix@C=15pt{0\ar[r]&\Lambda\ar[r]&T_{0}\ar[r]&T_{1}\ar[r]&0}$ with $T_{0}$, $T_{1}\in \add T$.
\end{itemize}}
\end{definition}

Let $T\in \modu \Lambda$. We use $\Gen T$ to denote the class of all modules $M$ in $\modu \Lambda$ generated by $T$, that is,
 \begin{align*}
 \Gen T=\{M\in\modu \Lambda:\text{there exists an integer }n \geq 0 \text { and} \\ \text{an epimorphism }\xymatrix@C=15pt{T^{n}\ar[r]&M\ar[r]&0} \text{in} \modu \Lambda \}.
 \end{align*}
  Dually, $\Cogen M$ is defined.

A tilting $\Lambda$-module $T$ induces a torsion pair $(\Gen T,T^{\perp_{0}})$ in $\modu \Lambda$, where
\begin{align*}
\Gen T=&\mathcal{T}(T):=\{X\in\modu \Lambda: \Ext_{\Lambda}^{1}(T,X)=0\}.
\end{align*}

Recall from \cite{HDRI96T} that a torsion pair $(\mathcal{T},\mathcal{F})$ in an abelian category $\mathcal{A}$ is called \emph{tilting} (resp. \emph{cotilting})
if any object in $\mathcal{A}$ is isomorphic to a subobject of an object in $\mathcal{T}$
(resp. a quotient object of an object in $\mathcal{F}$).

\begin{remark}\label{rem-tilting}
{\rm (\cite[Lemma \RNum{1}.3.1]{HDRI96T})}
 Let $\mathcal{A}$ be an abelian category and $(\mathcal{T},\mathcal{F})$ a torsion pair in $\mathcal{A}$. If $\mathcal{A}$ has enough injective objects, then $(\mathcal{T},\mathcal{F})$ is tilting if and only if $\mathcal{T}$ contains all injective objects. Dually, if $\mathcal{A}$ has enough projective objects, then $(\mathcal{T},\mathcal{F})$ is cotilting if and only if $\mathcal{F}$ contains all projective objects.
 \end{remark}

\begin{remark}\label{rem-module}
Let $T$ be a tilting $\Lambda$-module  in $\modu \Lambda$. Notice that $(\Gen T, T^{\perp_{0}})$ is a torsion pair in $\modu \Lambda$, by \cite[Theorem \RNum{6}.6.5]{AISD}, we have that $\Gen T$ contains all injective objects in $\modu \Lambda$. So, by Remark \ref{rem-tilting}, we have that $(\Gen T, T^{\perp_{0}})$ is a tilting torsion pair in $\modu \Lambda$.
\end{remark}

In case for abelian categories with enough injective objects or projective objects, the following result is a special case of \cite[Theorem 1(4)]{MXHZY}. The original proof is to check the definition directly by using the pushout or pullback  tools, however in this sequel, we give a different proof by using the equivalent characterizations of tilting or cotilting torsion pairs in Remark \ref{rem-tilting}.
\begin{proposition}\label{prop-inj-tilt}
Let $(\mathcal{A},\mathcal{B},\mathcal{C})$ be a recollement of abelian categories as (\ref{re}), and let $(\mathcal{T}',\mathcal{F}')$ and $(\mathcal{T}'',\mathcal{F}'')$ be torsion pairs in $\mathcal{A}$ and $\mathcal{C}$ respectively. Denote by
$(\mathcal{T},\mathcal{F})$ a glued torsion pair in $\mathcal{B}$ with respect to $(\mathcal{T'},\mathcal{F'})$ and $(\mathcal{T''},\mathcal{F''})$.
 Then we have the following statements.
 \begin{itemize}
 \item[(1)] Assume that $\mathcal{B}$ and $\mathcal{C}$ have enough injective objects.
     If $i^{!}$ and $j_{!}$ are exact, and if $(\mathcal{T}',\mathcal{F}')$ and $(\mathcal{T}'',\mathcal{F}'')$ are tilting, then $(\mathcal{T},\mathcal{F})$ is tilting.
 \item[(2)]  Assume that $\mathcal{B}$ and $\mathcal{C}$ have enough projective objects.
     If $i^{*}$ and $j_{*}$ are exact, and if $(\mathcal{T}',\mathcal{F}')$ and $(\mathcal{T}'',\mathcal{F}'')$ are cotilting, then $(\mathcal{T},\mathcal{F})$ is cotilting.
 \end{itemize}
\end{proposition}

\begin{proof}
Assume that $\mathcal{B}$ and $\mathcal{C}$ have enough projective (resp. injective) objects, then $\mathcal{A}$ has enough projective (resp. injective) objects by \cite[Proposition 2.5]{MXZTW}.

(1) Since ($\mathcal{T}',\mathcal{F}'$) and ($\mathcal{T}'',\mathcal{F}''$) are tilting torsion pairs in $\mathcal{A}$ and $\mathcal{C}$ respectively by assumption, we have that $\mathcal{T'}$ and $\mathcal{T''}$ contain all injective objects in $\mathcal{A}$ and $\mathcal{C}$ respectively by Remark \ref{rem-tilting}. Let $I$ be an injective object in $\mathcal{B}$. Since $j_{!}$ is exact, we get that $j^{*}$ preserves injective objects by \cite[Proposition 2.5]{MXZTW}. It follows that $j^{*}(I)$ is injective in $\mathcal{C}$ and $j^{*}(I)\in\mathcal{T''}$. Since $i^{!}$ is exact by assumption, by \cite[Lemma 2(2)]{MXHZY}, there exists an exact sequence
$$\xymatrix@C=15pt{0\ar[r]&i_{*}i^{!}(I)\ar[r]&I\ar[r]&
            j_{*}j^{*}(I)\ar[r]&0}$$
in $\mathcal{B}$. Notice that $i^{*}$ is right exact and $i^{*}j_{*}=0$ by Lemma \ref{lem-rec}(2)(4), applying $i^{*}$ to the above exact sequence yields an exact sequence
 $$\xymatrix@C=15pt{i^{!}(I)(\cong i^{*}i_{*}i^{!}(I))\ar[r]&i^{*}(I)\ar[r]&0}$$
in $\mathcal{A}$. Since $i^{!}$ preserves injectives by \cite[Proposition 2.5]{MXZTW}, we have that $i^{!}(I)$ is injective in $\mathcal{A}$ and $i^{!}(I)\in\mathcal{T'}$. It follows that $i^{*}(I)\in\mathcal{T'}$ since $\mathcal{T'}$ is closed under quotient objects. Thus $I\in\mathcal{T}$ and $\mathcal{T}$ contains all injective objects in $\mathcal{B}$, and hence $(\mathcal{T},\mathcal{F})$ is tilting by Remark \ref{rem-tilting}.

(2) It is similar to (1).
\end{proof}

From now on, let $\Lambda'$, $\Lambda$ and $\Lambda''$ be artin algebras such that there is a recollement as follows:
$$\xymatrix{\modu \Lambda'\ar[rr]!R|-{i_{*}}&&\ar@<-2ex>[ll]!R|-{i^{*}}
\ar@<2ex>[ll]!R|-{i^{!}}\modu \Lambda
\ar[rr]!L|-{j^{*}}&&\ar@<-2ex>[ll]!L|-{j_{!}}\ar@<2ex>[ll]!L|-{j_{*}}
\modu \Lambda''}.$$

Now we are in a position to prove our main result.
\begin{theorem}\label{rec-tilting}
 Let $T'$ and $T''$ be tilting modules in $\modu \Lambda'$ and $\modu \Lambda''$ respectively. If $i^{!}$ and $j_{!}$ are exact, then there exists a tilting $\Lambda$-module $T$ forming as $T:=j_{!}(T'')\oplus M$ for some $\Lambda$-module $M$, such that $(\mathcal{T},\mathcal{F})=(\Gen T,T^{\perp_{0}})$, where  $(\mathcal{T},\mathcal{F})$ is a glued torsion pair in $\modu \Lambda$ with respect to $(\Gen T', T'^{\perp_{0}})$ and $(\Gen T'', T''^{\perp_{0}})$.
 \end{theorem}
 \begin{proof}
 Let $e_{1}, e_{2}, \cdots,e_{n}$ be a basis of $\Ext^{1}_{\Lambda}(i_{*}(T'),j_{!}(T''))$. Represent each $e_{i}$ by an exact sequence
 $$\xymatrix@C=15pt{0\ar[r]&j_{!}(T'')\ar[r]&E_{i}\ar[r]&i_{*}(T')\ar[r]&0}$$
 in $\modu \Lambda$.
 Consider the following commutative diagram
 \begin{align*}
 \xymatrix@C=15pt{0\ar[r]&\bigoplus^{n} j_{!}(T'')\ar@{=}[d]\ar[r]&M\ar[d]^{h}\ar[r]&i_{*}(T')\ar[d]^{\theta}
 \ar[r]&0&\\
 0\ar[r]&\bigoplus^{n} j_{!}(T'')\ar[d]^{u_{i}}\ar[r]&{\mathop  \bigoplus \limits_{1 \le i \le {\rm{n}}}} E_{i}\ar[d]^{v_{i}}\ar[r]&\bigoplus^{n} i_{*}(T')\ar[d]^{k_{i}}\ar[r]&0\\
 0\ar[r]&j_{!}(T'')\ar[r]&E_{i}\ar[r]&i_{*}(T')\ar[r]&0.}
 \end{align*}
 Because $k_{i}\theta=1$, we get the following commutative diagram
  \begin{align*}
 \xymatrix@C=15pt{0\ar[r]&\bigoplus^{n} j_{!}(T'')\ar[d]^{u_{i}}\ar[r]&M\ar[d]^{v_{i}h}\ar[r]&i_{*}(T')
 \ar@{=}[d]\ar[r]&0&(*)\\
 0\ar[r]&j_{!}(T'')\ar[r]&E_{i}\ar[r]&i_{*}(T')\ar[r]&0,}
 \end{align*}
where the module $M$ is some kind of
universal extension (see \cite[Lemma \RNum{3}.6.1]{Happel}).

 Applying the functor $\Hom_{\Lambda}(-,j_{!}(T''))$ to the exact sequence $(*)$ yields an exact sequence
 \begin{align*}{\footnotesize
 \xymatrix@C=15pt{\cdots\ar[r]&\Hom_{\Lambda}(\bigoplus^{n} j_{!}(T''),j_{!}(T''))\ar[r]^-{\delta}&\Ext_{\Lambda}^{1}(i_{*}(T'),j_{!}(T''))
 \ar[r]&\Ext_{\Lambda}^{1}(M,j_{!}(T''))\ar[r]&\Ext_{\Lambda}^{1}(\bigoplus^{n} j_{!}(T''),j_{!}(T''))\ar[r]&\cdots}}
 \end{align*}
 with $\delta$ epic.
 Since $j_{!}$ is exact and $j_{!}$ preserves projective objects by assumption and \cite[Proposition 2.5]{MXZTW}, it follows from \cite[Proposition 2.8]{MXZTW} that
 $$ \Ext_{\Lambda}^{1}(j_{!}(T''),j_{!}(T''))\cong \Ext_{\Lambda''}^{1}(T'',j^{*}j_{!}(T''))\cong \Ext_{\Lambda''}^{1}(T'',T'')=0,$$
 and hence $\Ext_{\Lambda}^{1}(\bigoplus^{n} j_{!}(T''),j_{!}(T''))\cong \bigoplus^{n} \Ext_{\Lambda}^{1}(j_{!}(T''),j_{!}(T''))=0$.
  Therefore, $\Ext_{\Lambda}^{1}(M,j_{!}(T''))=0$ from the fact that $\delta$ is epic.

Taking $T:=j_{!}(T'')\oplus M$. We next prove that $T$ is the desired tilting $\Lambda$-module by dividing into four steps.

$\textbf{Step 1.}$ $T$ satisfies the condition $\textbf{(T1)}$:

Since $j_{!}$ is exact by assumption and $j_{!}$ preserves projectives by \cite[Proposition 2.5]{MXZTW}, $\pd _{\Lambda''}T''\leq 1$ implies $\pd _{\Lambda}j_{!}(T'')\leq 1$ and $\pd _{\Lambda}\bigoplus^{n}j_{!}(T'')\leq 1$.
 Since $i^{!}$ is exact by assumption, we have that $i_{*}$ preserves projectives by \cite[Proposition 2.6]{MXZTW}.
 Note that $i_{*}$ is exact by Lemma \ref{lem-rec}(2) and $\pd _{\Lambda'}T'\leq 1$, so $\pd _{\Lambda}i_{*}(T')\leq 1$.
 Then we have $\pd_{\Lambda}M\leq 1$ by the exact sequence $(*)$, and thus $\pd_{\Lambda}T\leq 1$.

$\textbf{Step 2.}$ $T$ satisfies the condition $\textbf{(T2)}$:

By the above arguments, we have
\begin{align*}
\Ext_{\Lambda}^{1}(T,T)&=\Ext_{\Lambda}^{1}( j_{!}(T'')\oplus M, j_{!}(T'')\oplus M)\\
&=\Ext_{\Lambda}^{1}(j_{!}(T''),j_{!}(T''))\oplus
\Ext_{\Lambda}^{1}(M,j_{!}(T''))
\oplus\Ext_{\Lambda}^{1}(j_{!}(T''),M)\oplus\Ext_{\Lambda}^{1}(M,M)\\
&=\Ext_{\Lambda}^{1}(j_{!}(T''),M)\oplus\Ext_{\Lambda}^{1}(M,M).
\end{align*}
Applying the exact functor $j^{*}$ to the exact sequence $(*)$ yields that $$j^{*}(M)\cong j^{*}({\bigoplus}^{n}j_{!}(T''))\cong {\bigoplus}^{n}j^{*}j_{!}(T'')\cong {\bigoplus} ^{n}T''.$$
Since $j_{!}$ is exact and $j_{!}$ preserves projectives  by assumption and \cite[Proposition 2.5]{MXZTW}, it follows from \cite[Proposition 2.8]{MXZTW} that $$\Ext_{\Lambda}^{1}(j_{!}(T''),M)\cong \Ext_{\Lambda''}^{1}(T'',j^{*}(M))\cong \Ext_{\Lambda''}^{1}(T'',{\bigoplus} ^{n}T'')=0.$$
On the other hand, applying the functors $\Hom_{\Lambda}(-,i_{*}(T'))$ and $\Hom_{\Lambda}(M,-)$ to the exact sequence $(*)$, we have the following exact sequences
$$
 \xymatrix@C=15pt{\cdots\ar[r]&\Ext_{\Lambda}^{1}(i_{*}(T'),i_{*}(T'))
 \ar[r]&\Ext_{\Lambda}^{1}(M,i_{*}(T'))\ar[r]&\Ext_{\Lambda}^{1}( \bigoplus^{n}j_{!}(T''),i_{*}(T'))\ar[r]&\cdots},$$
 and
 $$
  \xymatrix@C=15pt{\cdots\ar[r]&\Ext_{\Lambda}^{1}(M,\bigoplus^{n}j_{!}(T''))
 \ar[r]&\Ext_{\Lambda}^{1}(M,M)\ar[r]&\Ext_{\Lambda}^{1}( M,i_{*}(T'))\ar[r]&\cdots}.
$$
By \cite[Remark 3.7(1)]{PC14H}, we have $ \Ext_{\Lambda}^{1}(i_{*}(T'),i_{*}(T'))\cong \Ext_{\Lambda'}^{1}(T',T')=0$.
 Since $j_{!}$ is exact and preserves projectives, and $\Ima i_{*}=\Ker j^{*}$, it follows from \cite[Proposition 2.8]{MXZTW} that
$$ \Ext_{\Lambda}^{1}(j_{!}(T''),i_{*}(T'))\cong \Ext_{\Lambda''}^{1}(T'',j^{*}i_{*}(T'))=0.$$
Thus $\Ext_{\Lambda}^{1}(M,i_{*}(T'))=0$. Moreover, since $\Ext_{\Lambda}^{1}(M,\bigoplus^{n}j_{!}(T''))\cong \bigoplus^{n}\Ext_{\Lambda}^{1}(M,j_{!}(T''))=0$, we have $\Ext_{\Lambda}^{1}(M,M)=0$. So $\Ext_{\Lambda}^{1}(T,T)=0$.

Thus $T$ is a partial tilting module.

$\textbf{Step 3.}$ $(\mathcal{T},\mathcal{F})=(\Gen T,T^{\perp_{0}})$:

By \cite[Lemma \RNum{6}.2.3]{AISD}, we have that $T$ is $\Ext$-projective in $\Gen T$ and $(\Gen T,T^{\perp_{0}})$ is a torsion pair in $\modu \Lambda$.
For a $\Lambda$-module $Y$, by \cite[Lemma 2(2)]{MXHZY}, we have the following exact sequence
 $$\xymatrix@C=15pt{0\ar[r]&i_{*}i^{!}(Y)\ar[r]&Y\ar[r]&
j_{*}j^{*}(Y)\ar[r]&0}$$
in $\modu \Lambda$.
Applying the functor $\Hom_{\Lambda}(M,-)$ to the above exact sequence yields an exact sequence
\begin{align}\label{M-}
\xymatrix@C=15pt{0\ar[r]&\Hom_{\Lambda}(M,i_{*}i^{!}(Y))\ar[r]&\Hom_{\Lambda}(M,Y)
\ar[r]&\Hom_{\Lambda}(M,j_{*}j^{*}(Y))}.
\end{align}

Let $Y\in\mathcal{F}$, that is, $j^{*}(Y)\in T''^{\perp_{0}}$ and $i^{!}(Y)\in T'^{\perp_{0}}$. Then
\begin{align*}
\Hom_{\Lambda}(T,Y)=&\Hom_{\Lambda}(j_{!}(T'')\oplus M,Y)\\
=&\Hom_{\Lambda}(j_{!}(T''),Y)\oplus\Hom_{\Lambda}(M,Y)\\
\cong&\Hom_{\Lambda''}(T'',j^{*}(Y))\oplus \Hom_{\Lambda}(M,Y)\\
=&\Hom_{\Lambda}(M,Y).
\end{align*}
Since $i^{*}j_{!}=0$ and $i^{*}$ is right exact by Lemma $\ref{lem-rec}$(1)(2), applying the functor $i^{*}$ to the exact sequence $(*)$ yields $i^{*}(M)\cong i^{*}i_{*}(T')\cong T'$.

Since
$$\Hom_{\Lambda}(M,i_{*}i^{!}(Y))\cong \Hom_{\Lambda'}(i^{*}(M),i^{!}(Y))\cong \Hom_{\Lambda'}(T',i^{!}(Y))=0$$
 and
  $$\Hom_{\Lambda}(M,j_{*}j^{*}(Y))\cong \Hom_{\Lambda''}(j^{*}(M),j^{*}(Y))\cong\Hom_{\Lambda''}(\bigoplus {^{n}} T'',j^{*}(Y))=0,$$
 it follows that $\Hom_{\Lambda}(M,Y)=0$ from the exact sequence (\ref{M-}), and hence $\Hom_{\Lambda}(T,Y)=0$.
Thus $Y\in T^{\perp_{0}}$ and $\mathcal{F}\subseteq T^{\perp_{0}}$.

Conversely, let $Y\in T^{\perp_{0}}$, that is, $0=\Hom_{\Lambda}(T,Y)\cong \Hom_{\Lambda}(j_{!}(T''),Y)\oplus\Hom_{\Lambda}(M,Y)$, so $\Hom_{\Lambda}(j_{!}(T''),Y)=0$ and $\Hom_{\Lambda}(M,Y)=0$. It follows that $$\Hom_{\Lambda''}(T'',j^{*}(Y))\cong\Hom_{\Lambda}(j_{!}(T''),Y)=0$$
and $$\Hom_{\Lambda'}(T',i^{!}(Y))\cong\Hom_{\Lambda'}(i^{*}(M),i^{!}(Y))\cong \Hom_{\Lambda}(M,i_{*}i^{!}(Y))=0,$$
the last equality is obtained from the exact sequence (\ref{M-}).
Thus $j^{*}(Y)\in T''^{\perp_{0}}$ and $i^{!}(Y)\in T'^{\perp_{0}}$, and hence $Y\in\mathcal{F}$ and $T^{\perp_{0}}\subseteq\mathcal{F}$. Therefore,  $\mathcal{F}=T^{\perp_{0}}$ and $\mathcal{T}=\Gen T$, that is $(\mathcal{T},\mathcal{F})=(\Gen T,T^{\perp_{0}})$.

$\textbf{Step 4.}$ $T$ satisfies the condition $\textbf{(T3)}$:

By Remark \ref{rem-module}, $(\Gen T', T'^{\perp_{0}})$ and $(\Gen T'', T''^{\perp_{0}})$ are tilting torsion pairs in $\modu \Lambda'$ and $\Lambda''$ respectively. By Proposition \ref{prop-inj-tilt},  $(\mathcal{T},\mathcal{F})$ is a tilting torsion pair in $\modu \Lambda$,
and thus, by Remark \ref{rem-tilting},  $\mathcal{T}$ contains all injective objects in $\modu \Lambda$.
Hence, by \cite[Theorem \RNum{6}.6.5]{AISD},
 there exists a tilting module $\overline{T}=P(\mathcal{T})$ such that
 $$(\mathcal{T},\mathcal{F})=(\Gen \overline{T},\overline{T}^{\perp_{0}})=(\mathcal{T}(\overline{T}),\overline{T}^{\perp_{0}}).$$ Since $T$ is $\Ext$-projective in $\mathcal{T}$, we have $T\in \add \overline{T}$ and $\add T\subseteq \add \overline{T}$.
 Moreover, since $\overline{T}\in \mathcal{T}=\Gen T$, there is an exact sequence
 \begin{align}\label{3.2}
 \xymatrix@C=15pt{0\ar[r]&K\ar[r]&T^{n}\ar[r]&\overline{T}\ar[r]&0}
 \end{align}
in $\modu \Lambda$.

Let $X\in \mathcal{T}=\mathcal{T}(\overline{T})$.
Applying the functor $\Hom_{\Lambda}(-,X)$ to the exact sequence (\ref{3.2}) yields the following exact sequence
\begin{align*}
\xymatrix@C=15pt{\cdots\ar[r]&
\Ext^{1}_{\Lambda}(T^{n},X)
\ar[r]&\Ext^{1}_{\Lambda}(K,X)\ar[r]&\Ext^{2}_{\Lambda}(\overline{T},X)\ar[r]&\cdots}.
\end{align*}
Since $\pd \overline{T}\leq 1$ and $T$ is $\Ext$-projective in $\Gen T=\mathcal{T}$, we have $\Ext^{1}_{\Lambda}(K,X)=0$ and $K$ is $\Ext$-projective in $\mathcal{T}(\overline{T})$, it follows that $K\in \add \overline{T}$ from \cite[Theorem \RNum{6}.2.5]{AISD}.
Thus the exact sequence (\ref{3.2}) is split, so $\overline{T}\in \add T$ and $\add \overline{T}\subseteq \add T$.
Then $\add \overline{T}= \add T$, and it follows that $T$ satisfies the condition $(\textbf{T3})$.
Thus $T$ is a tilting $\Lambda$-module.
 \end{proof}

Now we show that the converse of Proposition \ref{prop-inj-tilt} holds true under certain conditions, which is also a special case of \cite[Proposition 1(2)]{MXHZY}. 
\begin{proposition}\label{tilting}
Let $(\mathcal{A},\mathcal{B},\mathcal{C})$ be a recollement of abelian categories as (\ref{re}), and let $(\mathcal{T},\mathcal{F})$ be a torsion pair in $\mathcal{B}$.
\begin{itemize}
\item[(1)] Assume that $\mathcal{B}$ and $\mathcal{C}$ have enough injective objects, and $(\mathcal{T},\mathcal{F})$ is tilting, we have
\begin{itemize}
\item[(1.1)] If $i^{*}$ is exact, then $(i^{*}(\mathcal{T}),i^{!}(\mathcal{F}))$ is a tilting torsion pair in $\mathcal{A}$.
\item[(1.2)] If $j_{*}j^{*}(\mathcal{F})\subseteq\mathcal{F}$, then $(j^{*}(\mathcal{T}),j^{*}(\mathcal{F}))$ is a tilting torsion pair in $\mathcal{C}$.
    \end{itemize}
    \item[(2)] Assume that $\mathcal{B}$ and $\mathcal{C}$ have enough projective objects, and $(\mathcal{T},\mathcal{F})$ is cotilting, we have
\begin{itemize}
\item[(2.1)] If $i^{!}$ is exact, then $(i^{*}(\mathcal{T}),i^{!}(\mathcal{F}))$ is a cotilting torsion pair in $\mathcal{A}$.
\item[(2.2)] If $j_{*}j^{*}(\mathcal{F})\subseteq\mathcal{F}$, then $(j^{*}(\mathcal{T}),j^{*}(\mathcal{F}))$ is a cotilting torsion pair in $\mathcal{C}$.
    \end{itemize}
 \end{itemize}
\end{proposition}
\begin{proof}
Assume that $\mathcal{B}$ and $\mathcal{C}$ have enough projective (resp. injective) objects, we have that $\mathcal{A}$ has enough projective (resp. injective) objects by
\cite[Proposition 2.5]{MXZTW}.

(1)
Since $(\mathcal{T},\mathcal{F})$ is tilting by assumption, by Remark \ref{rem-tilting}, we have that $\mathcal{T}$ contains all injective objects in $\mathcal{B}$.

(1.1) By \cite[Theorem 2(1)]{MXHZY},
 $(i^{*}(\mathcal{T}),i^{!}(\mathcal{F}))$ is a torsion pair in $\mathcal{A}$.
Let $I$ be an injective object in $\mathcal{A}$. Since $i_{*}$ preserves injectives by \cite[Proposition 2.6]{MXZTW},  $i_{*}(I)$ is injective in $\mathcal{B}$, and so $i_{*}(I)\in \mathcal{T}$. Thus $I\cong i^{*}i_{*}(I)\in i^{*}(\mathcal{T})$ and $i^{*}(\mathcal{T})$ contains all injective objects in $\mathcal{A}$,
and hence $(i^{*}(\mathcal{T}),i^{!}(\mathcal{F}))$ is tilting by Remark \ref{rem-tilting}.

(1.2) By \cite[Theorem 2(2)]{MXHZY},  $(j^{*}(\mathcal{T}),j^{*}(\mathcal{F}))$ is a torsion pair in $\mathcal{C}$. Let $I$ be an injective object in $\mathcal{C}$.
 Since $j_{*}$ preserves injectives by \cite[Proposition 2.5]{MXZTW}, we have $j_{*}(I)$ is injecitve in $\mathcal{B}$, and hence $j_{*}(I)\in \mathcal{T}$.
 Thus $I\cong j^{*}j_{*}(I)\in j^{*}(\mathcal{T})$ and $j^{*}(\mathcal{T})$ contains all injective objects in $\mathcal{C}$, and hence $(j^{*}(\mathcal{T}),j^{*}(\mathcal{F}))$ is tilting.

(2) It is similar to (1).
 \end{proof}

Now we show that the converse of Theorem \ref{rec-tilting} holds true under certain conditions, which reformulates \cite[Theorem 3.5]{MXHZYT}.

\begin{theorem}\label{con-thm-t}
Let $T$ be a tilting $\Lambda$-module and $(\mathcal{T},\mathcal{F}):=(\Gen T,T^{\perp_{0}})$ a torsion pair in $\modu \Lambda$ induced by $T$. Then we have
\begin{itemize}
\item[(1)] If $i^{*}$ is exact, then $i^{*}(T)$ is a tilting $\Lambda'$-module and $(i^{*}(\mathcal{T}),i^{!}(\mathcal{F}))=(\Gen i^{*}(T),(i^{*}(T))^{\perp_{0}})$.
\item[(2)] If $j_{*}j^{*}(\mathcal{F})\subseteq \mathcal{F}$, $j_{*}j^{*}(\mathcal{T})\subseteq\mathcal{T}$ and $j_{*}$ is exact,
then $j^{*}(T)$ is a tilting $\Lambda''$-module and $ (j^{*}(\mathcal{T}),j^{*}(\mathcal{F}))=(\Gen j^{*}(T),(j^{*}(T))^{\perp_{0}})$.
\end{itemize}
\end{theorem}
\begin{proof}
Since $(\mathcal{T},\mathcal{F})$ is a torsion pair induced by a tilting $\Lambda$-module $T$, we have that $\mathcal{T}$ contains all injective $\Lambda$-modules by \cite[Theorem \RNum{6}.6.5]{AISD}. It follows from Remark \ref{rem-tilting} that $(\mathcal{T},\mathcal{F})$ is a tilting torsion pair in $\modu \Lambda$.

(1) By Proposition \ref{tilting}(1), we have that $(i^{*}(\mathcal{T}),i^{!}(\mathcal{F}))$ is a tilting torsion pair in $\modu \Lambda'$, it follows that $i^{*}(\mathcal{T})$ contains all injective $\Lambda'$-modules from Remark \ref{rem-tilting}.
Since $i^{*}(T)\in i^{*}(\mathcal{T})$, we have that $\Gen i^{*}(T)\subseteq i^{*}(\mathcal{T})$ by the fact that $i^{*}(\mathcal{T})$ is closed under quotient objects.
On the other hand, it is clear that
$i^{*}(\mathcal{T})\subseteq \Gen i^{*}(\mathcal{T})$. Thus
$\Gen i^{*}(\mathcal{T})= i^{*}(\mathcal{T})$.

By \cite[Theorem \RNum{6}.6.5]{AISD}, there exists a tilting module $\overline{T'}=P(i^{*}(\mathcal{T}))$ such that
 $$(i^{*}(\mathcal{T}),i^{!}(\mathcal{F}))=(\Gen \overline{T'},\overline{T'}^{\perp_{0}})=(\mathcal{T}(\overline{T'}),\overline{T'}^{\perp_{0}}).$$
Notice that $i_{*}i^{*}(\mathcal{T})\subseteq \mathcal{T}$. By \cite[Proposition 2.8]{MXZTW}, we have
$$\Ext_{\Lambda'}^{1}(i^{*}(T),i^{*}(\mathcal{T}))\cong \Ext_{\Lambda}^{1}(T,i_{*}i^{*}(\mathcal{T}))=0,$$
that is, $i^{*}(T)$ is $\Ext$-projective in $i^{*}(\mathcal{T})$,  thus $i^{*}(T)\in \add \overline{T'}$ and $\add i^{*}(T)\subseteq \add \overline{T'}$.
Since $\overline{T'}\in i^{*}(\mathcal{T})=\Gen i^{*}(T)$, there exists an exact sequence
 \begin{align}\label{3.3}
 \xymatrix@C=15pt{0\ar[r]&K'\ar[r]&i^{*}(T)^{n}\ar[r]&\overline{T'}\ar[r]&0}
 \end{align}
in $\modu \Lambda'$.
Let $X'\in i^{*}(\mathcal{T})=\mathcal{T}(\overline{T'})$. Applying the functor $\Hom_{\Lambda'}(-,X')$ to the exact sequence (\ref{3.3}) yields the following exact sequence
\begin{align*}
\xymatrix@C=15pt{\cdots\ar[r]&
\Ext^{1}_{\Lambda'}(i^{*}(T)^{n},X')
\ar[r]&\Ext^{1}_{\Lambda'}(K',X')\ar[r]&\Ext^{2}_{\Lambda'}(\overline{T'},X')\ar[r]&\cdots}.
\end{align*}
Since $\pd \overline{T'}\leq 1$ and $i^{*}(T)$ is $\Ext$-projective in $i^{*}(\mathcal{T})$, we have that $\Ext^{1}_{\Lambda'}(K',X')=0$ and $K'$ is $\Ext$-projective in $\mathcal{T}(\overline{T'})$, and hence $K'\in \add \overline{T'}$ from \cite[Theorem \RNum{6}.2.5]{AISD}.
 Thus the exact sequence (\ref{3.3}) is split, which induces $\overline{T'}\in \add i^{*}(T)$ and $\add \overline{T'}\subseteq \add i^{*}(T)$.
Then we have that $\add \overline{T'}= \add i^{*}(T)$, and hence $i^{*}(T)$ is a tilting $\Lambda'$-module.

(2) Since $j_{*}j^{*}(\mathcal{F})\subseteq\mathcal{F}$ by assumption, $(j^{*}(\mathcal{T}),j^{*}(\mathcal{F}))$ is a tilting torsion pair in $\modu \Lambda''$ by Proposition \ref{tilting}(2).
By Remark \ref{rem-tilting}, we have that $j^{*}(\mathcal{T})$ contains all injective $\Lambda''$-modules.

Since $j^{*}(T)\in j^{*}(\mathcal{T})$ and $j^{*}(\mathcal{T})$ is closed under quotient objects, we have $\Gen j^{*}(T)\subseteq j^{*}(\mathcal{T})$. On the other hand, it is clear that $j^{*}(\mathcal{T})\subseteq \Gen j^{*}(T)$. Thus $\Gen j^{*}(T)= j^{*}(\mathcal{T})$.

By \cite[Theorem \RNum{6}.6.5]{AISD}, there exists a tilting $\Lambda''$-module $\overline{T''}=P(j^{*}(\mathcal{T}))$ such that $$(j^{*}(\mathcal{T}),j^{*}(\mathcal{F}))=(\Gen \overline{T''},{\overline{T''}}^{\perp_{0}})=(\mathcal{T}(\overline{T''}),
{\overline{T''}}^{\perp_{0}}).$$
Since $j_{*}j^{*}(\mathcal{T})\subseteq \mathcal{T}$ and $j_{*}$ is exact by assumption and \cite[Proposition 2.8]{MXZTW}, we have
$$\Ext_{\Lambda''}(j^{*}(T),j^{*}(\mathcal{T}))=\Ext_{\Lambda}(T,j_{*}j^{*}(\mathcal{T}))=0,$$
that is, $j^{*}(T)$ is $\Ext$-projective in $j^{*}(\mathcal{T})$. Thus $j^{*}(T)\in \add \overline{T''}$ and $\add j^{*}(T)\subseteq \add \overline{T''}$.

Since $\overline{T''}\in j^{*}(\mathcal{T})=\Gen j^{*}(T)$, there exists an exact sequence
 \begin{align}\label{3.4}
 \xymatrix@C=15pt{0\ar[r]&K''\ar[r]&j^{*}(T)^{m}\ar[r]&\overline{T''}\ar[r]&0}
 \end{align}
in $\modu \Lambda''$.
Let $X''\in j^{*}(\mathcal{T})=\mathcal{T}(\overline{T''})$. Applying the functor $\Hom_{\Lambda''}(-,X'')$ to the exact sequence (\ref{3.4}) yields the following exact sequence
\begin{align*}
\xymatrix@C=15pt{\cdots\ar[r]&
\Ext^{1}_{\Lambda''}(j^{*}(T)^{m},X'')
\ar[r]&\Ext^{1}_{\Lambda''}(K'',X'')\ar[r]&\Ext^{2}_{\Lambda''}(\overline{T''},X'')\ar[r]&\cdots}.
\end{align*}
Since $\pd \overline{T''}\leq 1$ and $j^{*}(T)$ is $\Ext$-projective in $j^{*}(\mathcal{T})$, we have that $\Ext^{1}_{\Lambda''}(K'',X'')=0$ and $K''$ is $\Ext$-projective in $\mathcal{T}(\overline{T''})$, it follows that $K''\in \add \overline{T''}$ from \cite[Theorem \RNum{6}.2.5]{AISD}.
 Thus the exact sequence (\ref{3.4}) is split, which shows $\overline{T''}\in \add j^{*}(T)$ and $\add \overline{T''}\subseteq \add j^{*}(T)$.
Then we have that $\add \overline{T''}= \add j^{*}(T)$, and thus $j^{*}(T)$ is a tilting $\Lambda''$-module.
\end{proof}

\begin{remark}
In fact, removing the condition ``$j_{*}j^{*}(\mathcal{F})\subseteq \mathcal{F}$" in Theorem \ref{con-thm-t}(2), $j^{*}(T)$ is still a tilting $\Lambda''$-module, but $(j^{*}(\mathcal{T}),j^{*}(\mathcal{F}))$ is not always a torsion pair in $\modu \Lambda''$ (see \cite[Theorem 3.3]{MXHZYT} and Example \ref{example}(3)).
\end{remark}

\section{Examples}

We give some examples to illustrate the obtained results.

In \cite{AIR},  Adachi,  Iyama and  Reiten introduced the notions of  support $\tau$-tilting
modules and its mutation. Using the mutation of support $\tau$-tilting
modules, one can compute all support $\tau$-tilting
modules  for a basic finite-dimensional algebra $\Lambda$. Note that a module $T\in\modu\Lambda$ is a tilting module if and only if it is a faithful $\tau$-tilting module, and
if $\Lambda$ is hereditary, then a module $T\in\modu\Lambda$ is a tilting module if and only if it is a $\tau$-tilting module (\cite{AIR}). Thus
it provides an efficient method for computing all tilting modules in $\modu\Lambda$.

Let $\Lambda',\Lambda''$ be artin algebras and $_{\Lambda'}N_{\Lambda''}$ an $(\Lambda',\Lambda'')$-bimodule, and let $\Lambda={\Lambda'\ {N}\choose \ 0\ \ \Lambda''}$ be a triangular matrix algebra.
Then any module in $\modu \Lambda$ can be uniquely written as a triple ${X\choose Y}_{f}$ with $X\in\modu \Lambda'$, $Y\in\modu \Lambda'$
and $f\in\Hom_{\Lambda'}(N\otimes_{\Lambda''}Y,X)$ (\cite[p.76]{AMRISSO95R}).

\begin{example}\label{example}
{\rm Let $\Lambda'$ be a finite dimensional algebra given by the quiver $\xymatrix@C=15pt{1\ar[r]&2}$ and $\Lambda''$ be a finite dimensional algebra given by the quiver $\xymatrix@C=15pt{3\ar[r]^{\alpha}&4\ar[r]^{\beta}&5}$ with the relation $\beta\alpha=0$. Define a triangular matrix algebra $\Lambda={\Lambda'\ \Lambda'\choose \ 0\ \ \Lambda''}$, where the right $\Lambda''$-module structure on $\Lambda'$ is induced by the unique algebra surjective homomorphsim $\xymatrix@C=15pt{\Lambda''\ar[r]^{\phi}&\Lambda'}$ satisfying $\phi(e_{3})=e_{1}$, $\phi(e_{4})=e_{2}$, $\phi(e_{5})=0$.  Then $\Lambda$ is
a finite dimensional algebra given by the quiver
$$
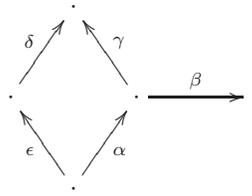
\xymatrix@C=15pt{&\cdot\\
\cdot\ar[ru]^{\delta}&&\ar[lu]_{\gamma}\cdot\ar[rr]^-{\beta}&&\cdot\\
&\ar[lu]^{\epsilon}\cdot\ar[ru]_{\alpha}}$$
with the relation $\gamma\alpha=\delta\epsilon$ and $\beta\alpha=0$. The Auslander-Reiten quiver of $\Lambda$ is
$$\xymatrix@C=15pt{{0\choose P(5)}\ar[rd]&&{S(2)\choose S(4)}\ar[rd]&&{S(1)\choose 0}\ar[rd]&&0\choose P(3)\ar[rd]\\
&{S(2)\choose P(4)}\ar[ru]\ar[rd]&&P(1)\choose S(4)\ar[ru]\ar[r]\ar[rd]&P(1)\choose P(3)\ar[r]&S(1)\choose P(3)\ar[ru]\ar[rd]&&{0\choose S(3)}.\\
S(2)\choose 0\ar[ru]\ar[rd]&&P(1)\choose P(4)\ar[ru]\ar[rd]&&0\choose S(4)\ar[ru]&&S(1)\choose S(3)\ar[ru]\\
&P(1)\choose 0\ar[ru]&&0\choose P(4)\ar[ru]}$$

By \cite[Example 2.12]{PC14H}, we have that
$$\xymatrix{\modu \Lambda'\ar[rr]!R|-{i_{*}}&&\ar@<-2ex>[ll]!R|-{i^{*}}
\ar@<2ex>[ll]!R|-{i^{!}}\modu \Lambda
\ar[rr]!L|-{j^{*}}&&\ar@<-2ex>[ll]!R|-{j_{!}}\ar@<2ex>[ll]!R|-{j_{*}}
\modu \Lambda''}$$
is a recollement of module categories, where
\begin{align*}
&i^{*}({X\choose Y}_{f})=\Coker f, & i_{*}(X)={X\choose 0},&&i^{!}({X\choose Y}_{f})=X,\\
&j_{!}(Y)={N\otimes_{\Lambda''} Y\choose Y}_{1}, & j^{*}({X\choose Y}_{f})=Y, &&j_{*}(Y)={0\choose Y}.
\end{align*}
\begin{itemize}
\item[(1)] Take tilting modules $T'=P(1)\oplus S(1)$ and $T''=P(5)\oplus P(4)\oplus P(3)$ in $\modu \Lambda'$ and $\modu \Lambda''$ respectively. They induce torsion pairs
    \begin{align*}
    (\mathcal{T'},\mathcal{F'})=&(\add(P(1)\oplus S(1)),\add S(2)),\\
(\mathcal{T''},\mathcal{F''})=&(\modu \Lambda'', 0)
\end{align*}
 in $\modu \Lambda'$ and $\modu \Lambda''$ respectively.
 Then by Theorem \ref{rec-tilting}, there is a tilting $\Lambda$-module $T={0\choose P(5)}\oplus{S(2)\choose P(4)}\oplus {P(1)\choose P(3)}\oplus{P(1)\choose P(4)}\oplus {P(1)\choose 0}$.
It induces a torsion pair
 \begin{align*}
(\Gen T,T^{\perp_{0}})=(\add ({P(1)\choose 0}\oplus{0\choose P(5)}\oplus{S(2)\choose P(4)}\oplus {P(1)\choose P(4)}\oplus{0\choose P(4)}\oplus{S(2)\choose S(4)}
\oplus {P(1)\choose S(4)}\oplus\ \ \ \\ {0\choose S(4)}\oplus {S(1)\choose 0}\oplus{S(1)\choose P(3)}
\oplus {S(1)\choose S(3)}
\oplus {P(1)\choose P(3)}\oplus{0\choose P(3)}\oplus {0\choose S(3)}),\add {S(2)\choose 0})
\end{align*}
 in $\modu \Lambda$,
which is exactly a glued torsion pair with respect to $(\mathcal{T'},\mathcal{F'})$ and $(\mathcal{T''},\mathcal{F''})$.

\item[(2)] Take  tilting modules $T'=P(1)\oplus S(1)$ and $T''=P(3)\oplus P(4)\oplus S(4)$ in $\modu \Lambda'$ and $\modu \Lambda''$ respectively.
They induce torsion pairs
    \begin{align*}
    (\mathcal{T'},\mathcal{F'})=&(\add(P(1)\oplus S(1)),\add S(2)),\\
(\mathcal{T''},\mathcal{F''})=&(\add (P(3)\oplus P(4)\oplus S(4)\oplus S(3)), \add P(5))
\end{align*}
 in $\modu \Lambda'$ and $\modu \Lambda''$ respectively.
  Then by Theorem \ref{rec-tilting}, there is a tilting $\Lambda$-module $T={P(1)\choose P(3)}\oplus {S(2)\choose P(4)}\oplus {S(2)\choose S(4)}\oplus{P(1)\choose P(4)}\oplus {P(1)\choose 0} $.
It induces a torsion pair
 \begin{align*}
(\Gen T,T^{\perp_{0}})=(\add ({S(2)\choose P(4)}\oplus {P(1)\choose P(4)}\oplus{P(1)\choose 0}\oplus{0\choose P(4)}\oplus{S(2)\choose S(4)}
\oplus {P(1)\choose S(4)}\oplus {0\choose S(4)}\oplus\ \ \ \\{S(1)\choose P(3)}
\oplus {S(1)\choose S(3)}
\oplus {S(1)\choose 0}\oplus {P(1)\choose P(3)}\oplus{0\choose P(3)}\oplus {0\choose S(3)}),\add ({S(2)\choose 0}\oplus {0\choose P(5)}))
\end{align*}
 in $\modu \Lambda$,
which is exactly a glued torsion pair with respect to $(\mathcal{T'},\mathcal{F'})$ and $(\mathcal{T''},\mathcal{F''})$.

\item[(3)] Take a tilting module $T={S(2)\choose S(4)}\oplus {P(1)\choose P(4)}\oplus {0\choose P(4)}\oplus {P(1)\choose S(4)}\oplus {P(1)\choose P(3)}$ in $\modu \Lambda$.
    It induces a torsion pair
    \begin{align*}
(\mathcal{T},\mathcal{F})=(\add({S(2)\choose S(4)}\oplus{P(1)\choose P(4)}\oplus {P(1)\choose S(4)}\oplus{0\choose P(4)}\oplus {S(1)\choose 0}\oplus {P(1)\choose P(3)}\oplus {0\choose S(4)}\oplus \\{S(1)\choose P(3)}\oplus {S(1)\choose S(3)}\oplus {0\choose S(3)}\oplus{0\choose P(3)},
\add({0\choose P(5)} \oplus {S(2)\choose 0}\oplus {P(1)\choose 0} \oplus {S(2)\choose P(4)}))
\end{align*}
    in $\modu \Lambda$.
By \cite[Theorem 3.5]{MXHZYT}, we have that $j^{*}(T)=S(4)\oplus S(4)\oplus P(4)\oplus P(4)\oplus P(3)$ is a tilting $\Lambda''$-module.
    But $j_{*}j^{*}(\mathcal{F})=\add ({0\choose P(4)}\oplus {0\choose P(5)})\subsetneq\mathcal{F}$, we have that
$$(j^{*}(\mathcal{T}),j^{*}(\mathcal{F}))=(\add (S(4)\oplus P(4)\oplus P(3)\oplus S(3)), \add P(5)\oplus P(4))$$
is not a torsion pair in $\modu \Lambda''$ and $(j^{*}(\mathcal{T}),j^{*}(\mathcal{F}))\neq (\Gen j^{*}(T),(j^{*}(T))^{\perp_{0}})$.

\item[(4)] Take a tilting module $T={S(2)\choose S(4)}\oplus {S(2)\choose P(4)}\oplus {P(1)\choose P(4)}\oplus {0\choose P(4)}\oplus {P(1)\choose P(3)}$ in $\modu \Lambda$.
    It induces a torsion pair
    \begin{align*}
(\mathcal{T},\mathcal{F})=(\add({S(2)\choose S(4)}\oplus{S(2)\choose P(4)}\oplus{P(1)\choose P(4)}\oplus {P(1)\choose S(4)}\oplus{0\choose P(4)}\oplus {S(1)\choose 0}\oplus {P(1)\choose P(3)}\oplus {0\choose S(4)}\oplus \\{S(1)\choose P(3)}\oplus {S(1)\choose S(3)}\oplus {0\choose S(3)}\oplus{0\choose P(3)},
\add({0\choose P(5)} \oplus {S(2)\choose 0}\oplus {P(1)\choose 0}))
\end{align*}
    in $\modu \Lambda$.
By Theorem \ref{con-thm-t}(2), we have that $j^{*}(T)=S(4)\oplus P(4)\oplus P(4)\oplus P(4)\oplus P(3)$ is a tilting $\Lambda''$-module and
$$(j^{*}(\mathcal{T}),j^{*}(\mathcal{F}))=(\add (S(4)\oplus P(4)\oplus P(3)\oplus S(3)), \add P(5))= (\Gen j^{*}(T),(j^{*}(T))^{\perp_{0}}).$$
\end{itemize}}
\end{example}

\vspace{0.5cm}
\textbf{Acknowledgement}.
The first author was supported by Henan University of Engineering (No. DKJ2019010).
The second author was supported by the NSF of China (Nos. 11971225, 11901341), the project ZR2019QA015 supported by Shandong Provincial Natural Science Foundation, and the Young Talents Invitation Program of Shandong Province. Part of this work was done by the second author during a visit at Tsinghua University.
He would like to express his gratitude to  Department of Mathematical Sciences and especially
to Professor Bin Zhu for the warm hospitality and the excellent working conditions.
The authors thank Professor Zhaoyong Huang for his careful guidance and consistent encouragement. Special thanks to
the referee whose valuable corrections and suggestions have  improved the presentation and organization of
this article.

\end{document}